\documentclass[12pt]{amsart}
\usepackage{a4}
\usepackage[T1]{fontenc}
\usepackage[all]{xy}
\usepackage{lipsum}
\usepackage{url}
\usepackage{tikz-network}
\usepackage{stackrel}
\usepackage{color}

\usepackage{amsmath,amsthm,amssymb}
\newcommand{\aspas}[1]{``{#1}''}
\usepackage{xcolor}
\usepackage{graphicx}
\usepackage{tikz}
\usetikzlibrary{
  hobby,
  intersections,
  spath3,
  decorations.markings,
  arrows.meta,
}

\newtheorem{theorem}{Theorem}[section]

\newtheorem{example}[theorem]{Example}
\newtheorem{proposition}[theorem]{Proposition}
\theoremstyle{definition}
\newtheorem{definition}[theorem]{Definition}

\newtheorem{remark}[theorem]{Remark}
\newtheorem{corollary}[theorem]{Corollary}

\numberwithin{equation}{section}

\begin{document}


\vspace{0.5in}

\renewcommand{\bf}{\bfseries}
\renewcommand{\sc}{\scshape}
\vspace{0.5in}

\title[Injective category number]%
{The injective category number on continuous maps}

\author[C. A. I. Zapata---R. Rabanal]{Cesar A. Ipanaque Zapata---Roland Rabanal}
\address{\textsc{Cesar A. Ipanaque Zapata}\\
Departamento de Matem\'atica, IME\\
Universidade de S\~ao Paulo\\
Rua do Mat\~ao 1010 CEP: 05508-090 S\~ao Paulo-SP, Brazil}
\email{cesarzapata@usp.br}

\thanks{The first author would like to thank grant\#2023/16525-7 and grant\#2022/16695-7, S\~{a}o Paulo Research Foundation (\textsc{fapesp}) for financial support.}

\address{\textsc{R. Rabanal}\\
Departamento de Ciencias,\\
Pontificia Universidad Cat\'{o}lica del Per\'{u}\\
Av. Universitaria 1801, San Miguel, 15088,  Per\'{u}}
\email{rrabanal@pucp.edu.pe}%
\thanks{The second author was partially supported by \textsc{pucp--Peru}~(\textsc{dgi}: 2023--E--0020).}



\begin{abstract} We introduce the concept of injective category number $\text{IC}(f)$ for a continuous map $f\colon X\to~Y$, and present fundamental results concerning this numerical invariant. The value $\text{IC}(f)$ quantifies the \aspas{complexity} or \aspas{categorical structure} underlying  the question: under what conditions is $f$ injective? More precisely, $\text{IC}(f)$ is the smallest positive integer $\ell$ such that $X$ can be covered by $\ell$ open subsets $U_1,\ldots,U_\ell$, with each restriction map $f_{\mid U}:U\to Y$ being injective. For instance, we examine the behaviour of $\text{IC}(f)$ under pullbacks and compositions of maps. In addition, we provide a cohomological lower bound for $\text{IC}(f)$. When  $f$ has a finite number of multiple points, we express $\text{IC}(f)$ in terms of these points of non-injectivity. In the case that $f$ is the quotient map $\mathfrak{q}^X:X\to X/G$, where $X$ is a  metric free $G$-space, we provide a lower bound for the injective category of $\mathfrak{q}^X$ in terms of the $2$-th index, $\text{ind}_2(X,G)$. When $G=\mathbb{Z}_2$, this lower bound is shown to be sharp. These results link a classical problem in Borsuk-Ulam theory to contemporary research developments in the study of injective category numbers. 
\end{abstract}
\maketitle


\section{Introduction}\label{secintro}%
In this paper ``space'' means a topological space, and by a ``map'' we will always mean a continuous map. 

\medskip Let $M$ and $N$ be $C^\infty$ manifolds and $f\colon M\to N$ be an immersion, that is, $f$ is $C^\infty$ and its differential $d f_p\colon T_pM \to T_{f(p)}N$ is injective, for any $p\in M$.
Consequently,  $f$ is locally injective: in other words, for each $p\in M$, there exists an open neighbourhood $U$ of $p$ such that the restriction map $f_{|U}:U\to N$ is injective, see \cite[p. 15]{guillemin1974differential}.
By other hand, it is well known that there are immersions which are not (globally) injective, even in the case of local diffeomorphisms between Euclidean spaces, \cite{MR1658243,MR1940235,MR2129721,MR2344172,MR1249211,MR2226492}.

\medskip Hence, it is natural to stay the following question: When is a locally injective map injective?

\medskip Motivated by this question we introduce the notion of injective category number of a map $f:X\to Y$, denoted by $\text{IC}(f)$ (Definition~\ref{defn:injective-category}), together with fundamental results concerning this numerical invariant. The value $\text{IC}(f)$ quantifies the \aspas{complexity} or \aspas{category} of the question above. More precisely, $\text{IC}(f)$ is the smallest positive integer $\ell$ such that $X$ can be covered by $\ell$ open subsets $U_1,\ldots,U_\ell$, with each restriction map $f_{\mid U}:U\to Y$ being injective. For instance, we have that $\text{IC}(f)=1$ if and only if $f$ is injective. Furthermore, suppose $\mathrm{IC}(f)=n$, then, for any cover $U_1,\ldots,U_{\ell}$ of $X$ by $\ell$ open sets, with $\ell<n$, there is at least one set $U_j$ containing two points $x\neq x'$ such that $f(x)=f(x')$.

\medskip The main results of this work are:
\begin{itemize}
    \item Introduction of the injective category number of a map $f:X\to Y$.
    \item Theorem~\ref{thm:pullback} shows that the injective category number is well behaved under pullbacks.
    \item In the case that $f$ admits a finite number of multiple points, we present the injective category of $f$ in terms of its multiple points (Theorem~\ref{thm:finite-multiple-points}).
    \item Theorem~\ref{thm:composite} presents inequalities of IC under composition.
    \item In the case that $f$ is a surjective map between manifolds of the same dimension, Theorem~\ref{thm:cohomo-ic} presents a cohomological lower bound for $\mathrm{IC}(f)$. In addition, Theorem~\ref{between-euclidean} presents a cohomological obstruction for the injectivity of a map with codomain an Euclidean space.
    \item In the case that $f$ is the quotient map $\mathfrak{q}^X:X\to X/G$, where $X$ is a metric free $G$-space, Theorem~\ref{thm:BU-implies-IC} presents a lower bound for the injective category of $\mathfrak{q}^X$ in terms of the $2$-th index $\text{ind}_2(X,G)$. In the case $G=\mathbb{Z}_2$, this lower bound is achieved.
\end{itemize} 

\medskip The present paper is organized in three sections, briefly described as follows.
In Section~\ref{sec:injective-category}, we introduce the notion of injective category number $\text{IC}(f)$ (Definition~\ref{defn:injective-category}) together with its fundamental results (Theorem~\ref{thm:pullback}, Theorem~\ref{thm:finite-multiple-points}, Theorem~\ref{thm:composite}). In addition, we present several examples of injective category number (Example~\ref{exam:immersion-r-r}, Example~\ref{exam:twists-circle}, Example~\ref{exam:rose}).
Furthermore, we present a cohomological lower bound for the injective category of a surjective map between topological manifolds with the same dimension (Theorem~\ref{thm:cohomo-ic}). In addition, we present a cohomological obstruction for the injectivity of a map with codomain an Euclidean space (Theorem~\ref{between-euclidean}).
In Section~\ref{sec:involution}, we present a review of the $2$-th Borsuk-Ulam property and $2$-th index of a free $G$-space, and provide a lower bound for the injective category number of the quotient map in terms of the $2$-th index. In the case $G=\mathbb{Z}_2$, this lower bound is achieved  (Theorem~\ref{thm:BU-implies-IC}). As a direct application, we show that $\mathrm{IC}(S^n\to\mathbb{R}P^n)=n+2$ (Example~\ref{exam:sphere}).

\section{Injective category number}\label{sec:injective-category}
In this section we introduce the notion of injective category number together with its fundamental results and several examples. 
Furthermore, we present a cohomological lower bound for the injective category of a surjective map between topological manifolds of the same dimension. In addition, we provide a cohomological obstruction for the injectivity of a map with codomain an Euclidean space. 

\subsection{Definitions and Examples} Let $f:X\to Y$ be a map. We say that $f$ is \textit{locally injective} if, for each point $x\in X$, there exists an open subset $U\subset X$ such that the restriction map $f_{\mid U}:U\to Y$ is injective. Equivalently, there exists an open cover $\{U_\lambda\}_{\lambda\in\Lambda}$ of $X$ such that each restriction map $f_{\mid U_\lambda}:U_\lambda\to Y$ is injective. Observe that if $f$ is locally injective, then $f^{-1}(y)$ has the discrete topology for each $y\in f(X)$.

\begin{example}
  Let $M$ and $N$ be $C^\infty$ manifolds and $f:M\to N$ be an immersion. Then $f$ is locally injective, see \cite[p. 15]{guillemin1974differential}\footnote{In contrast, any submersion $g:P\to Q$, between $C^\infty$ manifolds with $\dim(P)>\dim(Q)$, is not locally injective.}.  
\end{example} 

On the other hand, we have the following example.

\begin{example}\label{exam:non-locally-injective}
Consider the map $f:S^1\to\mathbb{R}^2$ given by \[f(x,y)=\begin{cases}
    (x,y),&\hbox{ if $y\geq 0$;}\\
    (x,-y),&\hbox{ if $y\leq 0$.}
\end{cases}\] Notice that for any open neighbourhood $U$ of $(1,0)$ we have that the restriction map $f_{\mid U}:U\to \mathbb{R}^2$ is not injective. 
Hence, such map $f$ is not locally injective.
\end{example}

As shown in the introduction, it is natural to stay the following question: When is a locally injective map injective? Motivated by this question we introduce the notion of injective category number.

\begin{definition}[Injective Category Number]\label{defn:injective-category}
Let $f:X\to Y$ be a map. The \textit{injective category number} of $f$, denoted by $\text{IC}(f)$, is the smallest positive integer $\ell$ such that there are open subsets $U_1,\ldots,U_\ell\subset X$ that $X=U_1\cup\cdots\cup U_\ell$ and each restriction map $f_{\mid U_j}:U_j\to Y$ is injective. We set $\text{IC}(f)=\infty$ if no such integer $\ell$ exists.     
\end{definition}

\medskip Note that, $\text{IC}(f)=1$ if and only if $f$ is injective. 
Furthermore, if $\text{IC}(f)<\infty$, then $f$ is locally injective. 

\begin{remark}\label{rem:locally-injective-ic-finite}
Let $f:X\to Y$ be a locally injective map, i.e., there exists an open cover $\{U_\lambda\}_{\lambda\in\Lambda}$ of $X$ such that each restriction map $f_{\mid U_\lambda}:U_\lambda\to Y$ is injective. In the case that $X$ is compact we have that there exist $\lambda_1,\ldots,\lambda_k\in\Lambda$ such that $\{U_{\lambda_i}\}_{i=1}^{k}$ is a cover of $X$, and thus $\text{IC}(f)\leq k<\infty$. The other implication is not true. For example, for $m\leq n$, the canonical immersion $i:\mathbb{R}^m\to\mathbb{R}^n$, $i(x_1,\ldots,x_m)=(x_1,\ldots,x_m,0,\ldots,0)$, satisfies $\text{IC}(i)=1$ and $\mathbb{R}^m$ is not compact. Also, the immersion $f:\mathbb{R}\setminus\{0\}\to\mathbb{R}$, given by  $f(x)=x^2$, satisfies $\text{IC}(f)=2$.  
\end{remark} 

We also have the following examples.  %

\begin{example}\label{exam:immersion-r-r}
    Any immersion $f:\mathbb{R}\to\mathbb{R}$ is injective and hence $\mathrm{IC}(f)=1$. Since $f$ is an immersion, we have that $f'(x)\neq 0$ for any $x\in\mathbb{R}$. Then, by the Mean Value Theorem, we obtain that $f'(x)> 0$ for any $x\in\mathbb{R}$ or $f'(x)< 0$ for any $x\in\mathbb{R}$ (here we use that $f'$ is a continuous map). Thus, $f$ is strictly increasing or decreasing and hence $f$ is injective.
\end{example}

\begin{example}\label{exam:twists-circle}
    Consider a map $f$ that twists the circle into a figure eight, see Figure~\ref{fig:eight}. This is an immersion of $S^1$ into $\mathbb{R}^2$ which is not injective (hence $\mathrm{IC}(f)\geq 2$). Here, $f^{-1}(y)=\{p_1,p_2\}$. We can find open subsets $U_1,U_2\subset S^1$ such that  for each $i=1,2$, $U_i\cap f^{-1}(y)=\{p_i\}$ (hence the restriction map $f_{\mid U_i}:U_i\to\mathbb{R}^2$ is injective) and $U_1\cup U_2=S^1$, see Figure~\ref{fig:eight-covering}. Then $\mathrm{IC}(f)\leq 2$ and therefore $\mathrm{IC}(f)= 2$. 
\end{example}

\begin{figure}[htb] 
 \centering
 \begin{tikzpicture}
 \node at (3.5,0.5) {$f$};
\Vertex[x=2,y=0,size=0.2,label=$p_2$,position=right,color=black]{F};
\Vertex[x=-2,y=0,size=0.2,label=$p_1$,position=left,color=black]{G};
\Vertex[x=5,y=0,size=0.2,label=$y$,position=below,color=black]{H};
\draw[->, >=latex](2.8,0)--(4.2,0); 
\draw[thick] (2,0) arc[start angle=0, end angle=360, radius=2];
\draw[thick] (6,1) arc[start angle=0, end angle=360, radius=1];
\draw[thick] (6,-1) arc[start angle=0, end angle=360, radius=1];
\end{tikzpicture}
 \caption{Figure eight.}
 \label{fig:eight}
\end{figure}

\begin{figure}[htb] 
 \centering
 \begin{tikzpicture}
 \node at (3.5,0.5) {$f$};
 \node at (-2.7,0) {$U_1$};
  \node at (2.7,0.5) {$U_2$};
\Vertex[x=2,y=0,size=0.2,label=$p_2$,position=left,color=black]{F};
\Vertex[x=-2,y=0,size=0.2,label=$p_1$,position=right,color=black]{G};
\Vertex[x=5,y=0,size=0.2,label=$y$,position=below,color=black]{H};
\draw[->, >=latex](2.8,0)--(4.2,0); 
\draw[thick] (2,0) arc[start angle=0, end angle=360, radius=2];
\draw[thick] (6,1) arc[start angle=0, end angle=360, radius=1];
\draw[thick] (6,-1) arc[start angle=0, end angle=360, radius=1];
\draw[thick, dashed] (1.6,1.6) arc[start angle=45, end angle=315, radius=2.3];
\draw[thick, dashed] (-1.8,-1.8) arc[start angle=225, end angle=495, radius=2.5];
\end{tikzpicture}
 \caption{The open subsets $U_1,U_2\subset S^1$.}
 \label{fig:eight-covering}
\end{figure}

\subsection{Under pullbacks} Note that, if the following diagram

\begin{eqnarray*}
\xymatrix{ \widetilde{M} \ar[rr]^{\varphi} \ar[dr]_{\widetilde{f}} & & M \ar[dl]^{f}  \\
        &  N & }
\end{eqnarray*}
commutes, where $\varphi$ is an injective map, then $\mathrm{IC}\hspace{.1mm}(\widetilde{f})\leq \mathrm{IC}\hspace{.1mm}(f)$. Suppose that $U\subset M$ is an open subset and the restriction map $f_{\mid}:U\to N$ is injective. Then, $V=\varphi^{-1}(U)\subset \widetilde{M}$ is an open subset and the restriction map $\widetilde{f}_{\mid}:V\to N$ is injective (note that $\widetilde{f}_{\mid}=f_{\mid}\circ\varphi_{\mid}$, where $\varphi_{\mid}:\varphi^{-1}(U)\to U$ is the restriction map which is a bijection).

\medskip Also, for any map $f:M\to N$ and any continuous map $\psi:X\to N$, note that any open subset $U\subset M$ such that $f_{\mid}:U\to N$ is injective induces an open subset $\pi_2^{-1}(U)\subset X\times_N M$ such that the restriction map ${\pi_1}_{\mid }:\pi_2^{-1}(U)\to X$ is injective, where $\pi_1:X\times_N M\to X$ is the canonical pullback, $X\times_N M=\{(x,m)\in X\times M:~\psi(x)=f(m)\}$ and $\pi_i$ is the $i$th coordinate projection. Suppose that $(x,m),(x^\prime,m^\prime)\in \pi_2^{-1}(U)$ with $\pi_1(x,m)=\pi_1(x^\prime,m^\prime)$ (hence $x=x^\prime$). Then \begin{align*}
    f(m)&=\psi(x)\\
    &=\psi(x^\prime)\\
    &=f(m^\prime).
\end{align*} Note that $m,m^\prime\in U$ and thus $m=m^\prime$ (here we use that $f_{\mid}:U\to N$ is injective). Therefore $(x,m)=(x^\prime,m^\prime)$ and so ${\pi_1}_{\mid }:\pi_2^{-1}(U)\to X$ is injective.

\begin{eqnarray*}
\xymatrix{ &X\times_N M \ar[rr]^{\pi_2} \ar[d]_{\pi_1} & & M \ar[d]^{f} & \\
       &X  \ar[rr]_{\psi} & &  N & \\
      }
\end{eqnarray*} Thus, \begin{align}\label{ineq-canonical}
    \text{IC}(\pi_1)&\leq \text{IC}(f).
\end{align} 

\medskip We present that the injective category number is well behaved under pullbacks.

\begin{theorem}[Under Pullback]\label{thm:pullback}
 Let $f:M\to N$ be a map. If the following square
\begin{eqnarray*}
\xymatrix{ \widetilde{M} \ar[r]^{\,\,} \ar[d]_{\widetilde{f}} & M \ar[d]^{f} & \\
       \widetilde{N} \ar[r]_{} &  N &}
\end{eqnarray*}
is a pullback, then $\mathrm{IC}\hspace{.1mm}(\widetilde{f})\leq \mathrm{IC}\hspace{.1mm}(f)$.    
\end{theorem}
\begin{proof}
Since $\widetilde{f}$ is a pullback, we have the following commutative triangle
\begin{eqnarray*}
\xymatrix{ \widetilde{N}\times_N M \ar[rr]^{\varphi} \ar[dr]_{\pi_1} & & \widetilde{M} \ar[dl]^{\widetilde{f}}  \\
        &  \widetilde{N} & }
\end{eqnarray*} where $\varphi$ is a homeomorfismo and thus $\text{IC}(\pi_1)\leq\text{IC}(\widetilde{f})$. Similarly, since $\pi_1$ is the canonical pullback, we have the following commutative triangle
\begin{eqnarray*}
\xymatrix{ \widetilde{M} \ar[rr]^{\varphi^{-1}} \ar[dr]_{\widetilde{f}} & &  \widetilde{N}\times_N M\ar[dl]^{\pi_1}  \\
        &  \widetilde{N} & }
\end{eqnarray*} and thus $\text{IC}(\widetilde{f})\leq\text{IC}(\pi_1)$. Hence, the equality $\text{IC}(\widetilde{f})=\text{IC}(\pi_1)$ holds. By the inequality~(\ref{ineq-canonical}), we obtain $\mathrm{IC}\hspace{.1mm}(\widetilde{f})\leq \mathrm{IC}\hspace{.1mm}(f)$.
\end{proof}

Let $f:M\to N$ be a map and $\varnothing\neq A\subset N$. Consider the restriction map $f_{\mid}:f^{-1}(A)\to A$. A direct consequence of  Theorem~\ref{thm:pullback} is the following statement. 

\begin{proposition}\label{prop:restriction}
  Let $f:M\to N$ be a map and $\varnothing\neq A\subset N$. We have \[\mathrm{IC}(f_{\mid})\leq\mathrm{IC}(f).\]  
\end{proposition}

Proposition~\ref{prop:restriction} implies the following remark.

\begin{remark}\label{rem:finteic-impliy-finitepre}
Let $f:X\to Y$ be a locally injective map. Observe that $|f^{-1}(y)|=\mathrm{IC}(f_{\mid}:f^{-1}(y)\to\{y\})$ because $f^{-1}(y)$ has the discrete topology. If $\mathrm{IC}(f)<\infty$, then $f^{-1}(y)$ is finite for any $y\in f(X)$. In fact, by Proposition~\ref{prop:restriction}, we obtain $|f^{-1}(y)|\leq \mathrm{IC}(f)$ for any $y\in f(X)$. The other implication does not hold (see Example~\ref{exam:sphere}(2) below).
\end{remark}


\subsection{Multiple points} Before to state the following statement we present the following definition. Let $f:M\to N$ be a locally injective map. We call a point $y\in f(M)$ a \textit{$k$th point} of $f$ if $f^{-1}(y)$ is a finite set with $k$ elements (recall that $f^{-1}(y)$ has the discrete topology). 
Notice that $k\geq1$. Furthermore, $f$ is injective if and only if $f$ has not $k$th points with $k\geq 2$ (that is, $f$ has only $1$th points). In this context, a $k$th point with $k\geq 2$ is called a \textit{multiple point} of $f$.   

\medskip Now, in the case that $f$ admits a finite number of multiple points, we present the injective category of $f$ in terms of its multiple points.

\begin{theorem}[Multiple Points]\label{thm:finite-multiple-points}
 Let $f:M\to N$ be a locally injective map and suppose that $y_1,\ldots,y_\ell\in f(M)$ are the only multiple points. Suppose that $M$ is a $T_1$ space. 
 Consider $k_i=| f^{-1}(y_i)|$ (that is, $y_i$ is a $k_i$th point) for each $i=1,\ldots,\ell$. Then \[\mathrm{IC}(f)=\max_{1\leq i\leq\ell}\{k_i\}.\] 
\end{theorem}
\begin{proof}
    For each $i\in\{1,\ldots,\ell\}$, by Remark~\ref{rem:finteic-impliy-finitepre}, we have $k_i\leq\mathrm{IC}(f)$. Hence, the inequality $\max_{1\leq i\leq\ell}\{k_i\}\leq \mathrm{IC}(f)$ always holds.  

    \par
    Now, we will check that $\mathrm{IC}(f)\leq\max_{1\leq i\leq\ell}\{k_i\}$. Suppose that $\max_{1\leq i\leq\ell}\{k_i\}=\tilde{M}$. 
    For each $i\in\{1,\ldots,\ell\}$,  assume that $f^{-1}(y_i)=\{x_1^i,\ldots,x_{k_i}^i\}$. 
    In this context, for each $i\in\{1,\ldots,\ell\}$, set \[I^i_{m}=\begin{cases}
     f^{-1}(y_i)\setminus\{x^i_m\},&\hbox{for $1\leq m\leq k_i$;}\\  
     f^{-1}(y_i),&\hbox{for $k_i< m\leq \tilde{M}$.}
    \end{cases}\] Then, for each $m\in\{1,\ldots,\tilde{M}\}$, consider \[U_m=M\setminus\bigcup_{i=1}^{\ell} I^i_m.\] Note that each $U_m$ is an open subset of $M$ (here we use that $M$ is a $T_1$ space), $M=\bigcup_{m=1}^{\tilde{M}}U_m$ and $f$ is injective over each $U_m$. Therefore, $ \mathrm{IC}(f)\leq \tilde{M}=\max_{1\leq i\leq\ell}\{k_i\}.$
\end{proof}

Theorem~\ref{thm:finite-multiple-points} implies the following statement.

\begin{corollary}\label{cor:only-type}
  Let $M$ be a $T_1$ space, and $f:M\to N$ be a locally injective map that admits only a finite number of multiple points of type $k$th. Then \[\mathrm{IC}(f)=k.\]    
\end{corollary}

\begin{example}\label{exam:rose}
   Let $f:S^1\to\mathbb{R}^2$ be an immersion (and, of course, it is a locally injective map) of the circle in the plane with a single multiple point of type $k$th, see Figure~\ref{fig:5-pet}. Then, by Corollary~\ref{cor:only-type}, we have that $\mathrm{IC}(f)=k$. 
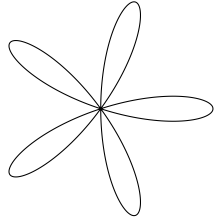
\begin{figure}[htb] 
 \centering
 \begin{tikzpicture}
\pgfmathsetmacro\R{sqrt(7/pi)}   
\draw[smooth] plot[domain=0:36*5,samples=200] (\x:{\R*cos(5*\x)}); 
\end{tikzpicture}
 \caption{Immersion of the circle in the plane with a single multiple point of type $5$th.}
 \label{fig:5-pet}
\end{figure}
\end{example}

\subsection{Under composition}
Let $f:M\to N$ and $g:N\to P$ be locally injective maps. We find that the composition $g\circ f:M\to P$ is also a locally injective map. In fact, for $x\in M$ we consider an open subset $U\subset M$ such that $x\in U$ and the restriction map $f_{\mid U}:U\to N$ is injective. Set $y=f(x)\in N$ and consider an open subset $V\subset N$ that $y\in V$ and the restriction map $g_{\mid V}:V\to P$ is injective. Take $W=U\cap f^{-1}(V)$ and note that $W$ is an open subset of $M$, $x\in W$ and the restriction map $(g\circ f)_{\mid W}:W\to P$ is injective. 

\medskip Furthermore, we present inequalities of IC under composition. 

\begin{theorem}[Under Composition]\label{thm:composite}
   Let $f:M\to N$ and $g:N\to P$ be maps. We have \begin{align}\label{ineq:composite}
     \mathrm{IC}(f)\leq\mathrm{IC}(g\circ f)\leq \mathrm{IC}(f)\cdot\mathrm{IC}(g).  
   \end{align} 
\end{theorem}
\begin{proof}
We will see the first inequality. Suppose that $U$ is an open subset of $M$ such that the restriction map $(g\circ f)_{\mid U}:U\to P$ is injective. We will see that the restriction map $f_{\mid U}:U\to N$ is injective. In fact, set $x,x^\prime\in U$ with $f(x)=f(x^\prime)$. Then $g(f(x))=g(f(x^\prime))$ and thus $x=x^\prime$ (here we use the fact that $g\circ f$ is injective over $U$). Thus, we conclude that $\mathrm{IC}(f)\leq\mathrm{IC}(g\circ f)$. 

Now, we will check the second inequality. Suppose that $U$ is an open subset of $M$ such that the restriction map $f_{\mid U}:U\to N$ is injective and $V$ is an open subset of $N$ such that the restriction map $g_{\mid V}:V\to P$ is injective. Consider $W_{U,V}=\left(f_{\mid U}\right)^{-1}(V)$ and observe that $W$ is an open subset of $M$ with the restriction map $\left(g\circ f\right)_{\mid W_{U,V}}:W_{U,V}\to P$ injective. Therefore, we have $\mathrm{IC}(g\circ f)\leq \mathrm{IC}(f)\cdot\mathrm{IC}(g)$.
\end{proof}

For any map $f:M\to N$, note that the equalities \[\mathrm{IC}(1_N\circ f)=\mathrm{IC}(f)=\mathrm{IC}(1_N)\cdot\mathrm{IC}(f)\] hold. Hence, in general, the inequalities in (\ref{ineq:composite}) cannot be improved.  

\medskip Let $f:M\to M$ be a self-map. For each $k\geq 1$, set $f^k=f\circ\cdots\circ f$ the $k$th iteration of $f$. Note that if $f$ is injective, then $f^k$ is also injective for any $k\geq 1$. Furthermore, we have the following statement.

\begin{proposition}\label{prop:iteration}
 Let $f:M\to M$ be a self-map. For each $k\geq 1$, we have \[\mathrm{IC}(f^k)\leq \mathrm{IC}(f^{k+1})\leq \left(\mathrm{IC}(f)\right)^{k+1}.\]   
\end{proposition}
\begin{proof}
    It follows from Theorem~\ref{thm:composite}.
\end{proof}

Proposition~\ref{prop:iteration} implies the following statement.

\begin{corollary}\label{cor:ic-finite-finite-any-itera}
Let $f:M\to M$ be a self-map. If $\mathrm{IC}(f)<\infty$, then $\mathrm{IC}(f^k)<\infty$ for any $k\geq 1$.    
\end{corollary}




\subsection{Cohomological lower bound} Let $M, N$ be topological manifolds (without boundary) with the same dimension. From the Invariance Domain Theorem (see \cite[Theorem 10.3.7, p. 251]{tom2008}) one has that if $f:M\to N$ is an inyective map, then $f:M\to N$ is an open map (and, of course,  $f(M)$ is an open subset of $N$ and $f:M\to f(M)$ is an homeomorphism), see \cite[Corol\'ario B.1.3, p. 133]{zapata2017}.

\medskip

We present a cohomological lower bound for the injective category number of a surjective map between topological manifolds with the same dimension, a tool widely used in computations, arises as follows. We follow the notation from \cite[Defini\c{c}\~{a}o 2.4.14, p. 60]{zapata2022} and \cite[Chapter 2, Section 6]{kono2006}, a multiplicative cohomology theory $h^\ast$ on the homotopy category of pairs of spaces comes equipped with a relative cohomology product \[\cup:h^\ast(X,A)\otimes h^\ast(X,B)\to h^\ast(X,A\cup B)\] whenever $A,B\subset X$ are excisive. In our case, $A$ and $B$ will be open sets. On the other hand, consider the \textit{index of nilpotence}
$$\text{nil}(S)=\min\{n:~\text{every product of $n$ elements in $S$ vanishes}\}$$
defined for a subset $S$ of a ring $R$.

\begin{theorem}[Cohomological Lower Bound]\label{thm:cohomo-ic}
Let $M, N$ be topological manifolds (without boundary) with the same dimension. Let $h^\ast$ be a multiplicative cohomology theory on the homotopy category of pairs of spaces. For any surjective map $f:M\to N$, we have
\[\mathrm{nil}\left(\mathrm{Ker}(f^\ast:h^\ast(N)\to h^\ast(M))\right)\leq\mathrm{IC}(f).\]
\end{theorem}
\begin{proof}
  Set $\mathrm{IC}(f)=m<\infty$ and take $U_1,\ldots,U_m$ open subsets of $M$ such that $M=U_1\cup\cdots\cup U_m$ and each restriction $f_{| U_i}:U_i\to N$ is injective (and, of course,  each $f(U_i)$ is an open subset of $N$ and $f_|:U_i\to f(U_i)$ is a homeomorphism). Note that $N=\bigcup_{i=1}^m f(U_i)$ (here we use that $f:M\to N$ is surjective). Suppose that $\mathrm{nil}\left(\mathrm{Ker}(f^\ast:h^\ast(N)\to h^\ast(M))\right)>m$. Then, there are $\alpha_1,\ldots,\alpha_m\in \mathrm{Ker}(f^\ast:h^\ast(N)\to h^\ast(M))$ such that $\alpha_1\cup\cdots\cup\alpha_m\neq 0$. Set $\alpha_i\in h^{n_i}(N)$ for each $i=1,\ldots,m$. 

\medskip We have the following commutative diagrams:
    \begin{eqnarray*} 
\xymatrix{ h^\ast(M)\ar[rd]_{incl^\ast}  & & h^\ast(N)\ar[ll]_{f^\ast}  \ar[dl]^{\left(f_{| U_i}\right)^\ast}  \\
        &  h^\ast(U_i) & } & \xymatrix{ h^\ast(N)\ar[rd]_{\left(f_{| U_i}\right)^\ast} \ar[rr]^{incl^\ast}  & & h^\ast(f(U_i)) \ar[dl]^{\left(f_|\right)^\ast}  \\
        &  h^\ast(U_i) & }
\end{eqnarray*} Then $\mathrm{Ker}\left(f^\ast\right)\subset \mathrm{Ker}\left(\left(f_{| U_i}\right)^\ast\right)$ and $\mathrm{Ker}\left(\left(f_{| U_i}\right)^\ast\right)=\mathrm{Ker}\left(incl^\ast:h^\ast(N)\to h^\ast(f(U_i))\right)$ (this last equality follows from the fact that $\left(f_|\right)^\ast$ is an isomorphism). Hence, $\alpha_1,\ldots,\alpha_m\in \mathrm{Ker}\left(incl^\ast:h^\ast(N)\to h^\ast(f(U_i))\right)$. 

\medskip Now, for each $i=1,\ldots,m$, we consider the exact sequence associated to the pair $(N,f(U_i))$:
    \[\cdots\to h^{n_i}(N,f(U_i))\stackrel{j^\ast}{\to}
    h^{n_i}(N)\stackrel{incl^\ast}{\to}h^{n_i}(f(U_i))
    \stackrel{\partial}{\to} h^{n_i+1}(X,f(U_i))\to\cdots\]
 where $j:N\hookrightarrow (N,f(U_i))$ is the inclusion map. Then $\alpha_i\in \mathrm{Ker}\left(incl^\ast\right)=\mathrm{Im}(j^\ast)$ and thus there is a cohomology class $\widetilde{a_i}\in h^{n_i}(N,f(U_i))$ such that $j^\ast(\widetilde{\alpha_i})=\alpha_i$.

\medskip Then $\widetilde{a_1}\cup\cdots\cup \widetilde{a_m}=0$ because $\widetilde{a_1}\cup\cdots\cup \widetilde{a_m}\in h^{n_1+\cdots+n_m}(N,N)=0$ (here we use the fact that $N=\bigcup_{i=1}^m f(U_i)$). Hence, one has $0=j^\ast\left(\widetilde{a_1}\cup\cdots\cup\widetilde{a_m}\right)=\alpha_1\cup\cdots\cup\alpha_m$, which is a contradiction. Therefore, $\mathrm{nil}\left(\mathrm{Ker}(f^\ast:h^\ast(N)\to h^\ast(M))\right)\leq m=\mathrm{IC}(f)$.
\end{proof}

\medskip

In concrete cases (e.g.~those worked out below) we do not attempt to compute the entire kernel of the homomorphism
$f^\ast:h^\ast(N)\to h^\ast(M)$, but we rather look for specific elements in the kernel and try to find long
non-trivial products.

\medskip 

\begin{remark}
    \noindent
    \begin{enumerate}
        \item The surjective hypothesis in Theorem~\ref{thm:cohomo-ic} cannot be relaxed. For example, consider an inclusion map $j:(0,1)\hookrightarrow S^1$. Set $H^\ast(-;\mathbb{Z})$ be singular cohomology with integer coefficients. Note that a generator $\alpha\in H^1(S^1;\mathbb{Z})\cong \mathbb{Z}$ satisfies $j^\ast(\alpha)=0$, and thus \[\mathrm{nil}\left(\mathrm{Ker}(j^\ast:H^\ast(S^1;\mathbb{Z})\to H^\ast((0,1);\mathbb{Z}))\right)= 2.\] However, $\mathrm{IC}(j)=1$.
        \item The lower bound in Theorem~\ref{thm:cohomo-ic} can be reached. For example, in the case that $f$ is a homeomorphism (in this case, one has $\mathrm{nil}\left(\mathrm{Ker}(f^\ast)\right)=1=\mathrm{IC}(f)$). Furthermore, consider the map $g:(0,1)\to S^1$ given by \[g(t)=\begin{cases}
            e^{i2\pi(2t)},& \mbox{if $0<t\leq 1/2$;}\\
            e^{i2\pi(1/2)(t-1/2)},& \mbox{if $1/2\leq t<1$.}
        \end{cases}\] In this case, one has $\mathrm{nil}\left(\mathrm{Ker}(g^\ast:H^\ast(S^1;\mathbb{Z})\to H^\ast((0,1);\mathbb{Z}))\right)=2=\mathrm{IC}(g)$.
    \end{enumerate}
\end{remark}

\medskip

We also consider the following example.

\medskip

\begin{example}\label{exam:lower-quotient}
Let $q:S^n\to \mathbb{R}P^n$ be the usual quotient map. Note that $q^\ast(u)=0$ where $u\in H^1(\mathbb{R}P^n;\mathbb{Z}_2)\cong\mathbb{Z}_2$ is the generator. Recall that $u^n\neq 0$ then, by Theorem~\ref{thm:cohomo-ic}, one has $\mathrm{IC}(q)\geq n+1$. We will improve this estimate in the next section (Example~\ref{exam:sphere}).
\end{example}

\medskip

Furthermore, we have the following statement which presents a cohomological obstruction for the injectivity of a map whose codomain is an Euclidean space.

\medskip

\begin{theorem}[Codomain an Euclidean Space]\label{between-euclidean}%
 Given $n\geq 1$ and $X$ is a space. Let $f:X\to \mathbb{R}^n$ be a map such that $f^{-1}(S^{n-1})\neq\varnothing$ is a compact space. Consider $f_|:f^{-1}(S^{n-1})\to S^{n-1}$ as the usual restriction map. Let $h^\ast$ be a multiplicative cohomology theory on the homotopy category of pairs of spaces. Set $u\in h^\ast(S^{n-1})$ with $u\neq 0$. If $\left(f_|\right)^\ast(u)=0$, then $f$ is not injective (and, of course, $\mathrm{IC}(f)\geq 2$).
\end{theorem}
\begin{proof}
    Suppose that $f$ is injective. One has $f_|:f^{-1}(S^{n-1})\to S^{n-1}$ is a bijection where $f^{-1}(S^{n-1})$ is a compact space, and thus $f_|$ is a homeomorphism. Hence, $\left(f_|\right)^\ast$ is an isomorphism, which is a contradiction. Therefore, $f$ is not injective.
\end{proof}

\medskip

Before to present the next example, let us mention that in \cite[Theorem 4]{MR4753913} the authors present a non-injective polynomial local diffeomorphism  $g:\mathbb{R}^2\to\mathbb{R}^2$ such that $g$ is surjective. 


\begin{example}\label{exam:coho-restrci-non-injective}
    Let $f:\mathbb{R}^2\to\mathbb{R}^2$ be a polynomial map. Suppose $f^{-1}(S^1)\neq\varnothing$ is a compact space. If $\left(f_|\right)^\ast(u)=0$ where $u\in H^1(S^1;\mathbb{Z})$ is a generator, then, by Theorem~\ref{between-euclidean}, $f$ is not injective (and, of course, $\mathrm{IC}(f)\geq 2$).
\end{example}


We have the following remark.

 \begin{remark}\label{rem:same-dim-locally-inject-opne}
 Let $M, N$ be topological manifolds (without boundary) with the same dimension, and $f:M\to N$ be a map.  \begin{enumerate}
     \item[(1)] From the Invariance Domain Theorem (see \cite[Theorem 10.3.7, p. 251]{tom2008}) one has that if $f:M\to N$ is locally inyective, then $f:M\to N$ is an open map. 
     \item[(2)] Suppose that $M$ is compact and $N$ connected and not compact. Note that $f:M\to N$ must not be locally injective (and, of course, $\text{IC}(f)=\infty$). Otherwise, $f:M\to N$ would be an open map (by Item(1)), and thus $f(M)$ would be open and compact (and of course closed), and these  imply that $f(M)=N$, because $N$ is connected, which is a contradiction.    
 \end{enumerate} 
 \end{remark}

 By Remark~\ref{rem:same-dim-locally-inject-opne}(2), we have that any map $f:S^m\to \mathbb{R}^m$ must not be locally injective (and, of course, $\text{IC}(f)=\infty$). In contrast, see Example~\ref{exam:sphere}, for the usual quotient map $\mathfrak{q}^{S^n}:S^n\to\mathbb{R}P^n$ we have $\mathrm{IC}(\mathfrak{q}^{S^n})=n+2$.
\section{$2$-th index of a free $G$-space}\label{sec:involution}
In this section we present a review of $2$-th Borsuk-Ulam property and the notion of $2$-th index of a free $G$-space (presented in \cite{zapata-dac}) and provide a lower bound for the injective category number of the orbit (or quotient) map in terms of the 
$2$-th index (Theorem~\ref{thm:BU-implies-IC})\footnote{In \cite{zapata-dac}, the authors present the notion of $q$-th BUP and $q$-th index, $\text{ind}_q(X,G)$, for any $2\leq q\leq |G|$, but for this work we only consider the case $q=2$ because the inequality $\text{ind}_q(X,G)\geq \text{ind}_{q+1}(X,G)$ always hold (see \cite[Remark 2.4(2)]{zapata-dac}).}. In the case $G=\mathbb{Z}_2$, this lower bound is achieved.   

\medskip The \textit{ordered configuration space} of $\ell$ distinct points on $Y$ (see \cite{fadell1962configuration}) is the topological space \[F(Y,\ell)=\{(y_1,\ldots,y_\ell)\in Y^\ell:~y_i\neq y_j \text{ for any $i\neq j$}\}\] topologised as a subspace of the Cartesian power $Y^\ell=Y\times\cdots\times Y$ ($\ell$ times).

\medskip Let $G$ be a finite group of order $\ell:=|G|>1$. Set $G=\{g_1,\ldots,g_\ell\}$ and $h:G\to [\ell]$ a bijection given by $h(g_j)=j$, here $[\ell]=\{1,\ldots,\ell\}$. We consider the left action of $G$ on $Y^\ell$  by \begin{equation}\label{eq:action}
g(y_1,y_2,\ldots,y_\ell)=\left(y_{h(g^{-1}g_1)},\ldots,y_{h(g^{-1}g_\ell)}\right). \end{equation} This action restricts to a free left action of $G$ on $F(Y,\ell)$, see \cite{zapata-dac}. In the remainder of this paper we will consider $F(Y,\ell)$ as a $G$-space with respect to this action. 

\medskip  Let $X$ be a free $G$-space for $G$ a finite group of order $|G|>1$, and let $Y$ be a Hausdorff space. We say that $\left((X,G);Y\right)$ \textit{satisfies the $2$-th Borsuk-Ulam property} (which we shall routinely abbreviate to $2$-th BUP) if  for every map $f:X\to Y$ there exists a point $x\in X$ such that there exist distinct $g,g'\in G$ such that $f(gx)=f(g'x)$, see \cite[Remark 2.1(2)]{zapata-dac} (for the case $G=\mathbb{Z}_2$ also see \cite{zapata2023} and cf. \cite[p. 371]{crabb2016}). 

\medskip Let $S^m$ be the $m$-dimensional sphere, $A:S^m\to S^m$ the antipodal involution (i.e., $A(x)=-x$ for any $x\in S^m$) 
 and $\mathbb{R}^n$ the $n$-dimensional Euclidean space. The famous Borsuk-Ulam theorem states that for every continuous map $f:S^m\to \mathbb{R}^m$ there exists a point $x\in S^m$ such that $f(x)=f(-x)$ \cite{borsuk1933}. 

 \medskip A natural generalization of the Borsuk-Ulam theorem consists in replacing $S^m$ together with the free involution given by the antipodal map by a free $G$-space $X$, and then to ask which triples $\left((X,G);\mathbb{R}^n\right)$ satisfy the $2$-th BUP,  see \cite{zapata-dac} (for the case $G=\mathbb{Z}_2$ also see \cite{zapata2023} and cf. \cite[p. 372]{crabb2016}). A major problem is to find the greatest $n$ such that the $2$-th BUP holds for a specific $(X,G)$. 

\begin{definition}[$2$-th Index]\label{defn-index} Let $X$ be a free $G$-space for $G$ a finite group of order $\ell:=|G|>1$. Following \cite[Definition 2.3]{zapata-dac} the \textit{$2$-th index} of $(X,G)$, denoted by $\text{ind}_2(X,G)$, is defined by the least integer $n\in \{0,1, 2, \ldots\}$ such that there exists a $G$-equivariant map $X\to F(\mathbb{R}^{n+1},\ell)$. We set $\text{ind}_2(X,G)=\infty$ if no such $n$ exists.  
\end{definition}

From \cite[Remark 2.4(1)]{zapata-dac} (for the case $G=\mathbb{Z}_2$ also see \cite[Proposition 2.2]{goncalves2010} or \cite{zapata2023}) follows that the greatest $n\geq 0$ such that 
 $\left((X,G);\mathbb{R}^n\right)$ satisfies the $2$-th BUP coincides with $\text{ind}_2(X,G)$. Hence, observe that $\text{ind}_2(X,G)=\infty$ if and only if $\left((X,G);\mathbb{R}^n\right)$ satisfies the $2$-th BUP for any $n\geq 0$. The index, $\text{ind}_{2}(X,\mathbb{Z}_2)$,  coincides with the $\mathbb{Z}_2$-index from \cite[Definition 2.2, p. 41]{zapata2023}, cf. \cite[Definition 5.3.1]{matousek2003}. Note that the existence of a free action of $\mathbb{Z}_2$ on $X$ is equivalent to that of a fixed-point free involution $\tau:X\to X$. In this case, we write $\text{ind}(X,\tau)$ instead of $\text{ind}_{2}(X,\mathbb{Z}_2)$.

 \medskip Before to present the main theorem of this section, we state the following statement. 

 \begin{proposition}\label{prop:bup-covering}
   Let $X$ be a metric free $G$-space for $G$ a finite group of order $|G|>1$. Consider the following statements:
   \begin{itemize}
       \item[(a)] The triple $\left((X,G);\mathbb{R}^n\right)$ satisfies the $2$-th BUP.
       \item[(b)] For any cover $F_1,\ldots,F_{n+1}$ of $X$ by $n+1$ closed sets, there
is at least one set $F_j$ containing the set $\{gx,g'x\}$ for some $x\in X$ and for some distinct $g,g'\in G$. 
\item[(c)] For any cover $U_1,\ldots,U_{n+1}$ of $X$ by $n+1$ open sets, there
is at least one set $U_j$ containing the set $\{gx,g'x\}$ for some $x\in X$ and for some distinct $g,g'\in G$. 
   \end{itemize}
   We have:
   \begin{enumerate}
       \item[(1)] (a) $\Rightarrow$ (b).
\item[(2)] (b) $\Rightarrow$ (c). 
   \end{enumerate}
 \end{proposition}
 \begin{proof}
The proof of this claim proceeds by analogy with \cite[Theorem 2.1.1, p. 23]{matousek2003}. \begin{enumerate}
       \item[(1)] For a closed cover $F_1,\ldots,F_{n+1}$ of the metric space $X$ we define a map $f:X\to\mathbb{R}^n$ by $f(x)=\left(\mathrm{dist}(x,F_1),\ldots,\mathrm{dist}(x,F_n)\right)$, and we consider a point $x\in X$ with $f(gx)=f(g'x)=y$ for some distinct $g,g'\in G$, which exist by (a). If 
the $i$th coordinate of the point $y$ is $0$, then both $gx$ and $g'x$ are in $F_i$. If all coordinates of $y$ are nonzero, then both $gx$ and $g'x$ lie in $F_{n+1}$.
\item[(2)] Since $X$ is a metric space, it is a normal space. We will check $(b)\Rightarrow (c)$. It follows from the fact that for every open cover $U_1,\ldots,U_{n+1}$ of $X$ there exists a closed cover $F_1,\ldots,F_{n+1}$ of $X$ satisfying $F_i\subset U_i$ for $i=1,\ldots,n+1$ (here we use the fact that $X$ is normal, see \cite[Theorem 6.1, p. 152]{dugundji1966general}).  
   \end{enumerate}    
 \end{proof}
 
\medskip Let $X$ be a metric free $G$-space. We provide a lower bound for the injective category of the quotient map $\mathfrak{q}:X\to X/G$ in terms of the $2$-th index, $\text{ind}_2(X,G)$. In the case $G=\mathbb{Z}_2$, this lower bound is achieved. Note that, given a nonempty subset $A\subset X$, the restriction map $\mathfrak{q}_{\mid A}:A\to X/G$ is injective if and only if $\{gx,g'x\}\not\subset A$ for all $x\in X$ and for all distinct $g,g'\in G$, which is equivalent to $A\cap gA=\varnothing$ for any $g\in G\setminus\{1\}$.  

\begin{theorem}[IC and $2$-th Index]\label{thm:BU-implies-IC}
Let $X$ be a metric free $G$-space for $G$ a finite group of order $\ell:=|G|>1$. Let $\mathfrak{q}^X:X\to X/G$ be the quotient map, and $n=\mathrm{ind}_2(X,G)$. We have \[n+2\leq \mathrm{IC}(\mathfrak{q}^X)\leq \mathrm{IC}\left(\mathfrak{q}^{F(\mathbb{R}^{n+1},\ell)}\right).\] The equalities hold whenever $G=\mathbb{Z}_2$.   
\end{theorem}
\begin{proof}
Observe that, if $\mathrm{ind}_2(X,G)=\infty$, then $\mathrm{IC}(\mathfrak{q}^X)=\mathrm{IC}\left(\mathfrak{q}^{F(\mathbb{R}^{\infty},\ell)}\right)=\infty$ (by Proposition~\ref{prop:bup-covering} and note that, for cohomological reasons--analogously to \cite[Example 3.23, p. 41]{zapata2023}, $\left((F(\mathbb{R}^{\infty},\ell),G);\mathbb{R}^n\right)$ satisfies the $2$-BUP for any $n\geq 0$). Now, we suppose $n=\mathrm{ind}_2(X,G)<\infty$. We have that the triple $\left((X,G);\mathbb{R}^n\right)$ satisfies the $2$-th BUP. Suppose that there are $A_1,\ldots,A_{n+1}$ open subsets of $X$ such that $X=A_1\cup\cdots\cup A_{n+1}$. Then, by Proposition~\ref{prop:bup-covering}, there is at least one set $A_j$ containing $\{gx,g'x\}$ for some $x\in X$ and for some distinct $g,g'\in G$. Then the restriction map $\mathfrak{q}_{\mid A_j}:A_j\to X/G$ is not injective. Hence, we obtain $\mathrm{IC}(\mathfrak{q})\geq n+2=\mathrm{ind}_2(X,G)+2$. 
  
  On the other hand, we also have that there exists a $G$-equivariant map $\varphi:X\to F(\mathbb{R}^{n+1},\ell)$ and thus the following diagram: 
  \begin{eqnarray*}
\xymatrix{ X \ar[rr]^{\varphi} \ar[d]_{\mathfrak{q}^X} & & F(\mathbb{R}^{n+1},\ell)\ar[d]^{\mathfrak{q}^{F(\mathbb{R}^{n+1},\ell)}} & \\
       X/G \ar[rr]_{\overline{\varphi}} & & F(\mathbb{R}^{n+1},\ell)/G &}
\end{eqnarray*} is a pullback. It implies that $\mathrm{IC}(\mathfrak{q}^X)\leq \mathrm{IC}\left(\mathfrak{q}^{F(\mathbb{R}^{n+1},\ell)}\right)$ (by Theorem~\ref{thm:pullback}).

  \medskip  Now, for $G=\mathbb{Z}_2$, observe that there are $\mathbb{Z}_2$-equivariant maps $F(\mathbb{R}^{n+1},2)\to S^n$ and $S^n\to F(\mathbb{R}^{n+1},2)$, and thus $\mathrm{IC}\left(\mathfrak{q}^{F(\mathbb{R}^{n+1},2)}\right)=\mathrm{IC}(\mathfrak{q}^{S^n})$ (by Theorem~\ref{thm:pullback}). We will check the inequality $\mathrm{IC}(q^{S^n})\leq n+2$. From \cite[p. 24]{matousek2003}, we can conclude that there exists a covering of $S^n$ by open subsets $W_1,\ldots,W_{n+2}$ such that no $W_i$ contains a pair of antipodal points (to see this, we consider an $(n+1)$-simplex in $\mathbb{R}^{n+1}$ containing $0$ in its interior, and we obtain closed sets $F_1,\ldots,F_{n+2}$ by projecting the facets centrally from $0$ on $S^n$. Hence, for each $i=1,\ldots,n+2$, we consider $W_i$ as a small open neighbourhood of $F_i$ such that $W_i$ does not admit a pair of antipodal points). Hence, we see that the inequality $\mathrm{IC}(q^{S^n})\leq n+2$ is valid. 
\end{proof}

Theorem~\ref{thm:BU-implies-IC} provides a generalization of the Lyusternik-Schnirelmann-Borsuk Theorem to metric free $G$-spaces.

\begin{corollary}[Generalization of LSB Theorem for metric free $G$-spaces]
Let $X$ be a metric free $G$-space for $G$ a finite group of order $|G|>1$. Suppose $\mathrm{ind}_2(X,G)\geq n$. If $U_1,\ldots,U_{n+1}$ is an open cover of $X$, then there
is at least one set $U_j$ containing the set $\{gx,g'x\}$ for some $x\in X$ and for some distinct $g,g'\in G$.     
\end{corollary}

We have the following example.

\begin{example}\label{exam:sphere}
\noindent\begin{enumerate}
    \item[(1)] Let $n\geq 0$ and $\mathfrak{q}^{S^n}:S^n\to\mathbb{R}P^n$ be the natural quotient map. From the Borsuk-Ulam theorem, we have $\mathrm{ind}(S^n,A)=n$. Then, by Theorem~\ref{thm:BU-implies-IC}, we obtain $\mathrm{IC}(\mathfrak{q}^{S^n})=\mathrm{IC}\left(\mathfrak{q}^{F(\mathbb{R}^{n+1},2)}\right)=n+2$.      
    \item[(2)] When $S^\infty$ is the infinite dimensional sphere together with the antipodal involution, we have $\mathrm{ind}(S^\infty,A)=\infty$ (see \cite[Example 3.23, p. 41]{zapata2023}). Hence, by Theorem~\ref{thm:BU-implies-IC}, we conclude $\mathrm{IC}(\mathfrak{q}^{S^\infty})=\infty$. 
\end{enumerate} 
\end{example}

Example~\ref{exam:sphere}(1) improves the lower bound obtained in Example~\ref{exam:lower-quotient}. On the other hand, note that any point in $\mathbb{R}P^n$ is a 2th point of $\mathfrak{q}^{S^n}:S^n\to\mathbb{R}P^n$. Example~\ref{exam:sphere}(1) also shows that the finite condition on the number of multiple points of the Corollary~\ref{cor:only-type} cannot be relaxed.

\medskip Also, Theorem~\ref{thm:BU-implies-IC} implies the following result.

\begin{corollary}\label{cor:for-conf}
 Let $G$ be a finite group of order $\ell:=|G|>1$. We have \[\mathrm{ind}_2\left(F(\mathbb{R}^{n+1},\ell),G\right)+2\leq \mathrm{IC}\left(\mathfrak{q}^{F(\mathbb{R}^{n+1},\ell)}\right).\] The equality holds whenever $G=\mathbb{Z}_2$.   
\end{corollary}

 We present some connections of this work with recent research.

\begin{remark}
Let $X$ be a free $G$-space and $\mathfrak{q}^X:X\to X/G$ be the the quotient map. \begin{enumerate}
    \item[(1)] In \cite{zapata-dac}, Zapata and Gon\c{c}alves show a connection between the $q$-th index and sectional category theory. For instance, in \cite{zapata2023}, as stated in version 4 of the Arxiv version, or \cite{zapata-dac}, the authors provide that the index of $(X,\tau)$ coincides with the sectional category of $\mathfrak{q}^X$ minus 1 for any paracompact space $X$. Then, by Theorem~\ref{thm:BU-implies-IC}, we see that the injective category number coincides with the sectional category plus one. Recall that any metric space is paracompact (by the A.H. Stone’s Theorem). 
        \item[(2)] In \cite[Definition 6.2.3, p. 150]{matousek2003}, the author defined a different notion of $G$-index, given by $\mathrm{ind}_G(X):=\min\{d:~\text{there exists a $G$-map } X\to E_dG\}$. Here $E_dG$ is a finite classifying space of dimension $d$ for $G$ (in the sense of \cite[Definition 2.4, p. 1380]{martinez2024} or \cite[Definition 6.2.1, p. 149]{matousek2003}). Since $F(\mathbb{R}^{d+1},\ell)$ is $(d-1)$-connected where $\ell:=|G|>1$, there exists a $G$-map $E_dG\to F(\mathbb{R}^{d+1},\ell)$ (by \cite[Lemma 6.2.2, p. 150]{matousek2003}). Hence, we observe that the following inequality \[\mathrm{ind}_2(X,G)\leq\mathrm{ind}_G(X)\] always holds and the equality holds whenever $G=\mathbb{Z}_2$. This $\mathrm{ind}_G(X)$ has been used in \cite[Theorem 3.4, p. 1383]{martinez2024} to estimate $\mathrm{IC}(\mathfrak{q}^X)$ (observe that $\mathrm{IC}(\mathfrak{q}^X)$ coincides with the notion of open $G$-covering number of $X$ presented in \cite[Definition 3.2, p. 1383]{martinez2024}). Specifically, the author shows that the following inequalities \begin{equation}\label{eq:martinez}
            k+|G|\leq \mathrm{IC}(\mathfrak{q}^K)\leq k\left(|G|-1\right)+2
        \end{equation} always hold for any finite group $G$ and any geometric $G$-simplicial complex $K$ with $k:=\mathrm{ind}_G(K)\geq 1$. The gap between \cite[Theorem 3.4, p. 1383]{martinez2024} and Theorem~\ref{thm:BU-implies-IC} is real because metric free $G$-spaces may not be modeled simplicially. In particular, we cannot apply inequalities~(\ref{eq:martinez}) for the $G$-spaces $F(\mathbb{R}^{d+1},\ell)$ with $\ell=|G|$. Also, in \cite[Conjecture 3.5, p. 1383; Conjecture 3.6(2), p. 1384]{martinez2024} was conjectured that $\mathrm{IC}(\mathfrak{q}^{E_dG})=|G|+d$ and $\mathrm{IC}(\mathfrak{q}^{E_dG})<\mathrm{IC}(\mathfrak{q}^{E_{d+1}G})$ hold for any finite group $G$ and all $d\geq 0$. On the other hand, note that $\mathrm{IC}(\mathfrak{q}^{E_dG})\leq \mathrm{IC}(\mathfrak{q}^{F(\mathbb{R}^{d+1},\ell)})$ (by Theorem~\ref{thm:pullback}), and that the equality holds whenever $G=\mathbb{Z}_2$.    
\end{enumerate} 
\end{remark}

Finally, we propose the following future work.

\begin{remark}[Future Work]\label{rem:future-work}
 \noindent\begin{enumerate}
     \item[(1)] As mentioned before the Example~\ref{exam:coho-restrci-non-injective}, in \cite[Theorem 4]{MR4753913} the authors present a non-injective polynomial local diffeomorphism  $g:\mathbb{R}^2\to\mathbb{R}^2$ such that $g$ is surjective. Compute $\mathrm{IC}(g)$.  
     \item[(2)]  Compute $\mathrm{ind}_G(F(\mathbb{R}^{d+1},\ell))$, $\mathrm{ind}_2(F(\mathbb{R}^{d+1},\ell),G)$ and $\mathrm{IC}(\mathfrak{q}^{F(\mathbb{R}^{d+1},\ell)})$ for any finite group $G$ with order $\ell:=|G|\geq 3$. Observe that, by Definition~\ref{defn-index}, the inequality $\mathrm{ind}_2(F(\mathbb{R}^{d+1},\ell),G)\leq d$ always holds. 
 \end{enumerate}   
\end{remark}



\bibliographystyle{plain}

\end{document}